\documentclass[12pt]{amsart}

\pdfoutput=1
\usepackage[utf8]{inputenc}
\usepackage{tikz, calc}
\usepackage[justification=centering]{caption}
\usepackage{subcaption}
\usepackage{amssymb, amsthm, fullpage}
\usepackage[maxbibnames=99]{biblatex}
\usepackage{graphicx}
\usepackage{soul}
\usepackage{nopageno}
\usepackage[margin=1.5in]{geometry}

\usepackage{thmtools}
\usepackage{thm-restate}

\newtheorem{lemma}{Lemma}
\newtheorem{theorem}{Theorem}

\newcommand{\grid}[3][step=1cm, gray, very thin,]{
    \draw[#1] (-0.9,-0.9) grid (#2 + 0.9, #3 + 0.9);
}

\newcommand{\vertex}[4][black,]{
    \filldraw[#1] (#2, #3) circle (3pt) node[anchor=west]{#4};
}

\newcommand{\tower}[5][thick, black]{
    \draw (#2,#3-#4) -- (#2+#4,#3);
    \draw (#2+#4,#3) -- (#2,#3+#4);
    \draw (#2,#3+#4) -- (#2-#4,#3);
    \draw (#2-#4,#3) -- (#2,#3-#4);
    \vertex{#2}{#3}{#5};
}

\begin{document}


\title{Asymptotically Optimal Bounds for $(t,2)$ broadcast Domination on Finite Grids}

\author{Timothy W. Randolph}
\address{Department of Mathematics and Statistics, Williams College, United States}
\email{twr2@williams.edu}

\date{\today}

\begin{abstract}
Let $G=(V,E)$ be a graph and $t,r$ be positive integers. The \emph{signal} that a tower vertex $T$ of signal strength $t$ supplies to a vertex $v$ is defined as $sig(T,v)=max(t-dist(T,v),0),$ where $dist(T,v)$ denotes the distance between the vertices $v$ and $T$. In 2015 Blessing, Insko, Johnson, and Mauretour defined a \emph{$(t,r)$ broadcast dominating set}, or simply a \emph{$(t,r)$ broadcast}, on $G$ as a set $\mathbb{T}\subseteq V$ such that the sum of all signals received at each vertex $v \in V$ from the set of towers $\mathbb{T}$ is at least $r$. The $(t,r)$ broadcast domination number of a finite graph $G$, denoted $\gamma_{t,r}(G)$, is the minimum cardinality over all $(t,r)$ broadcasts for $G$.

Recent research has focused on bounding the $(t,r)$ broadcast domination number for the $m \times n$ grid graph $G_{m,n}$. In 2014, Grez and Farina bounded the $k$-distance domination number for grid graphs, equivalent to bounding $\gamma_{t,1}(G_{m,n})$. In 2015, Blessing et al. established bounds on $\gamma_{2,2}(G_{m,n})$, $\gamma_{3,2}(G_{m,n})$, and $\gamma_{3,3}(G_{m,n})$. In this paper, we take the next step and provide a asymptotically optimal upper bound on $\gamma_{t,2}(G_{m,n})$ for all $t>2$. We also prove the conjecture of Blessing et al. that their bound on $\gamma_{3,2}(G_{m,n})$ is optimal for large values of $m$ and $n$.
\end{abstract}

\maketitle

\section{Introduction}

Let $u$ and $v$ be two vertices of the connected graph $G=(V,E)$. By $dist(u,v)$ we denote the distance between $u$ and $v$, which is defined as the length of the shortest path between $u$ and $v$ or 0 if $u=v$. A dominating set for G is a subset $S \subseteq V$ such that for every $v \in V$, there exists a vertex $s \in S$ with $dist(v,s) \leq 1$. The domination number $\gamma(G)$ of a graph $G$ is the cardinality of the smallest dominating set on $G$. In 1992, Chang \cite{chang1992domination} established the following bound on the domination number of the $m \times n$ grid graph $G_{m,n}$ with $m, n > 8$

\begin{equation}
    \gamma(G_{m,n}) \leq \Bigl\lfloor \frac{(n+2)(m+2)}{5}\Bigr\rfloor - 4.
    \label{chang_conj}
\end{equation}

Goncalves, Pinlou, Rao and Thomasse \cite{gonccalves2011domination} proved in 2011 that equality holds in Equation~\ref{chang_conj} when $n \geq m \geq 16$. For a historical perspective on the development of domination theory and related problems, the reader is referred to the work of Haynes, Hedetniemi, and Slater~\cite{haynes1998fundamentals}. 

As the field of domination theory grows, questions arise concerning generalizations of the domination number. One generalization considers the $k$-distance dominating set for a graph $G = (V, E)$, defined as a subset $S \subseteq V$ such that for every $v \in V$, there exists a vertex $s \in S$ such that $dist(v,s) \leq k$. We denote the $k$-distance domination number as $\gamma_k(G)$. In 2014, Grez and Farina \cite{grez2014new} succeeded in bounding the $k$-distance domination number of the $m \times n$ grid graph as follows

\begin{equation}
    \gamma_k(G_{m,n}) \leq \Bigl\lfloor \frac{(m+2k)(n+2k)}{(2k^2+2k+1)}\Bigr\rfloor - 4.
    \label{grez_conj}
\end{equation}

In 2014, Blessing, Insko, Johnson and Mauretour \cite{Insko} defined $(t,r)$ broadcast domination for positive integers $t,r$. In this setting, given a graph $G=(V,E)$ and a positive integer $t$, we say that a \emph{tower vertex} $T$ supplies a signal equal to $sig(T, v) = max\{t - dist(T,v), 0\}$ to each vertex $v \in V$. A $(t,r)$ broadcast is a set $S \subseteq V$ of tower vertices such that the sum of all signals supplied to each vertex is at least $r$. The $(t,r)$ broadcast domination number of a graph $G$, denoted $\gamma_{t,r}(G)$, is defined as the minimal cardinality among all $(t,r)$ broadcasts on $G$. We note that $(t,r)$ broadcast domination is a natural generalization of domination and $k$-domination, as $\gamma(G) = \gamma_{2,1}(G)$ and $\gamma_k(G) = \gamma_{k+1,1}(G)$.

In their work, Blessing et al. studied the $(t,r)$ broadcast domination number of grid graphs for the $(t,r)$ pairs in the set $\{(2,2), (3,1), (3,2), (3,3)\}$. In particular, they established the following bounds on $(t,2)$ broadcast domination numbers for grid graphs

\begin{equation}
    \gamma_{2,2}(G_{m,n}) \leq \Bigl\lfloor \frac{(m+2)(n+2)}{3}\Bigr\rfloor - 5.
    \label{insko_bound_1}
\end{equation}

\begin{equation}
    \gamma_{3,2}(G_{m,n}) \leq \Bigl\lfloor \frac{(m+2)(n+2)}{8}\Bigr\rfloor - 1.
    \label{insko_bound_2}
\end{equation}

In this paper, we establish an asymptotically optimal upper bound for $\gamma_{t,2}(G_{m,n})$ when $t > 2$.

\begin{restatable}{theorem}{mainthm}
If $G_{m,n}$ is the grid graph with dimensions $m \times n$, and $t>2$, then
\[\gamma_{t,2}(G_{m,n}) \leq \left \lfloor \frac{(m+2(t-2))(n+2(t-2))}{2(t-1)^2} \right \rfloor.\]
\label{thm_main}
\end{restatable}

In addition, we prove the following lower bound on $\gamma_{t,2}(G_{m,n})$, which is a corollary of a result of Drews, Harris, and Randolph \cite{DHR}.

\begin{restatable}{theorem}{thmlb}
If $G_{m,n}$ is the grid graph with dimensions $m \times n$, and $t>2$, then
\[\gamma_{t,2}(G_{m,n}) \geq \frac{mn}{2(t-1)^2}.\]
\label{thm_2}
\end{restatable}
The upper and lower bounds we establish are separated by a gap linear in $m$ and $n$, and thus converge as $m$ and $n$ increase. In addition, Theorem \ref{thm_2} confirms the conjecture of Blessing et al. that Equation  \ref{insko_bound_2} is asymptotically optimal.

\section{Bounding $\gamma_{t,2}(G_{m,n})$}

In 2017, Drews, Harris, and Randolph \cite{DHR} established the optimal $(t,2)$ broadcast on $G_\infty$, the infinite graph
\[(\{v_{(i,j)} \, | \, i, j \in \mathbb{Z}\}, \{ (v_{(i,j)}, v_{(i+1,j)}), (v_{(i,j)}, v_{(i,j+1)}) \, | \, i, j \in \mathbb{Z} \}).\]
Because no finite set is a $(t,r)$ broadcast on $G_\infty$, the authors define an optimal $(t,r)$ broadcast as a $(t,r)$ broadcast with minimal broadcast density. The \emph{broadcast density} of an infinite $(t,r)$ broadcast $S$ is defined as 
\[lim_{n \to \infty} \frac{|S \cap V_{n \times n}|}{|V_{n \times n}|},\]
where $V_{n \times n}$ is an $n \times n$ grid of vertices anchored by a fixed vertex in $G_\infty$. Intuitively, the broadcast density captures the proportion of the vertices of $G_\infty$ contained in an infinite $(t,r)$ broadcast.

In this section, we construct a $(t,2)$ broadcast on the grid graph $G_{m,n}$ by transforming a subset of an optimal $(t,2)$ broadcast on $G_\infty$. We then demonstrate that our upper bound on $\gamma_{t,2}(G_{m,n})$ and the upper bound of Blessing et al. on $\gamma_{3,2}(G_{m,n})$ are converge to the optimal value as $m$ and $n$ increase, a result that follows from the optimality of the original $(t,2)$ broadcast on $G_\infty$. To begin our analysis we recall the following theorem of Drews et al. \cite{DHR}.

\begin{theorem}
\label{thm_t2}
If $t>2$, the optimal broadcast density of a $(t,2)$ broadcast on $G_\infty$ is 
\[\gamma_{t,2}(G_{\infty}) = \frac{1}{2(t-1)^2}.\] 
\end{theorem}

The \emph{broadcast outline} of a tower vertex $T$ is defined as the diamond shape formed by connecting the vertices of the set $\{v : dist(T,v) = t-1\}$ when a grid graph $G_{m,n}$ is embedded in $\mathbb{Z} \times \mathbb{Z}$. Drews et al. proved that any set of vertices whose broadcast outlines form a tiling of $\mathbb{Z} \times \mathbb{Z}$ is an optimal $(t,2)$ broadcast.

\begin{figure}[h!]
\centering
    \begin{subfigure}{0.5\textwidth}
    \begin{tikzpicture}[scale=0.55]
        \clip (-1, -1) rectangle (11, 6);
        \grid[lightgray]{10}{5}
        \tower{-1}{3}{2}{}
        \tower{1}{5}{2}{}
        \tower{0}{0}{2}{}
        \tower{2}{2}{2}{}
        \tower{4}{4}{2}{}
        \tower{6}{6}{2}{}
        \tower{3}{-1}{2}{}
        \tower{5}{1}{2}{}
        \tower{7}{3}{2}{}
        \tower{9}{5}{2}{}
        \tower{8}{0}{2}{}
        \tower{10}{2}{2}{}
        \tower{11}{-1}{2}{}
    \end{tikzpicture}
    \caption{An optimal $(3,2)$ broadcast on $G_\infty$ with offset tiling. }
    \end{subfigure}%
    \begin{subfigure}{0.5\textwidth}
    \begin{tikzpicture}[scale=0.55]
        \clip (-1, -1) rectangle (11, 6);
        \grid[lightgray]{10}{5}
        \tower{0}{4}{2}{}
        \tower{2}{6}{2}{}
        \tower{0}{0}{2}{}
        \tower{2}{2}{2}{}
        \tower{4}{4}{2}{}
        \tower{6}{6}{2}{}
        \tower{4}{0}{2}{}
        \tower{6}{2}{2}{}
        \tower{8}{4}{2}{}
        \tower{10}{6}{2}{}
        \tower{8}{0}{2}{}
        \tower{10}{2}{2}{}
        \tower{12}{0}{2}{}
    \end{tikzpicture}
    \caption{An optimal $(3,2)$ broadcast on $G_\infty$ with rectilinear tiling. }
    \end{subfigure}%
    
	\caption{ Optimal $(3,2)$ broadcasts on the infinite grid. }
    \label{fig:t2_tilings}
\end{figure}
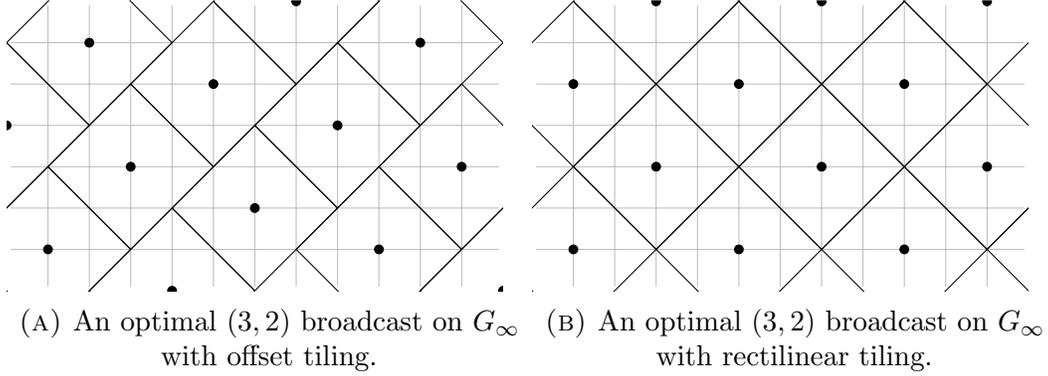

Figure \ref{fig:t2_tilings} illustrates two tilings created by the broadcast outlines of optimal $(3,2)$ broadcasts embedded in $\mathbb{Z} \times \mathbb{Z}$. Figure \ref{fig:t2_tilings}b illustrates an optimal $(t,2)$ broadcast in which the broadcast outlines are aligned to form a diagonal lattice, hereafter referred to as a \emph{rectilinear broadcast}. 

To establish our bound on $\gamma_{t,2}(G_{m,n})$, we describe how a subset of the rectilinear $(t,2)$ broadcast on $G_\infty$ can be transformed into a $(t,2)$ broadcast on $G_{m,n}$. Counting the number of elements in this subset gives an upper bound on $\gamma_{t,2}(G_{m,n})$. 

We first prove the special case when $G_{m,n}$ is a path, in which case $m=1$ or $n=1$.

\begin{lemma}
Let $G$ be a path of length $m$, and $t$ an integer greater than 2. Then
\[\gamma_{t,2}(G) \leq \left \lfloor \frac{(m+2(t-2))(1+2(t-2))}{2(t-1)^2} \right \rfloor.\]
\label{lemma:linegraph}
\end{lemma}
\begin{proof}
We first observe that paths of length up to $2(t-1)-1$ can be dominated by a single vertex, as illustrated in Figure \ref{fig:linegraph}a. Likewise, two vertices at a distance of $2(t-1)$ suffice to dominate a path of length up to $4(t-1)-1$. In general, $k$ vertices spaced at intervals of $2(t-1)$ constitute a $(t,2)$ broadcast for paths of length up to $2k(t-1)-1$. Figure \ref{fig:linegraph}b illustrates such a broadcast in the $k=3$, $t=4$ case.

\begin{figure}[h!]
\centering
    \begin{subfigure}{0.3\textwidth}
    \begin{tikzpicture}[scale=0.53]
        \clip (-1, -1) rectangle (6, 6);
        \grid[lightgray]{4}{4}

        \vertex[red]{0}{2}{}
        \vertex[red]{4}{2}{}
        \draw[red, very thick] (0,2) -- (4,2);
        
        \tower{2}{2}{3}{}
        
        \node[red] at (1.5,2.6) {\Large $G_{5,1}$};
    \end{tikzpicture}
    \caption{ $\gamma_{4,2}(G_{5,1}) = 1$. }
    \end{subfigure}%
    \begin{subfigure}{0.7\textwidth}
    \begin{tikzpicture}[scale=0.53]
        \clip (-1, -1) rectangle (18, 6);
        \grid[lightgray]{16}{4}
        
        \vertex[red]{0}{2}{}
        \vertex[red]{16}{2}{}
        \draw[red, very thick] (0,2) -- (16,2);
        
        \tower{2}{2}{3}{}
        \tower{8}{2}{3}{}
        \tower{14}{2}{3}{}
        
        \node[red] at (1.5,2.6) {\Large $G_{17,1}$};
    \end{tikzpicture}
    \caption{ $\gamma_{4,2}(G_{17,1}) = 3$. }
    \end{subfigure}%
    
	\caption{ $(4,2)$ broadcasts on $G_{m,1}$. }
    \label{fig:linegraph}
\end{figure}
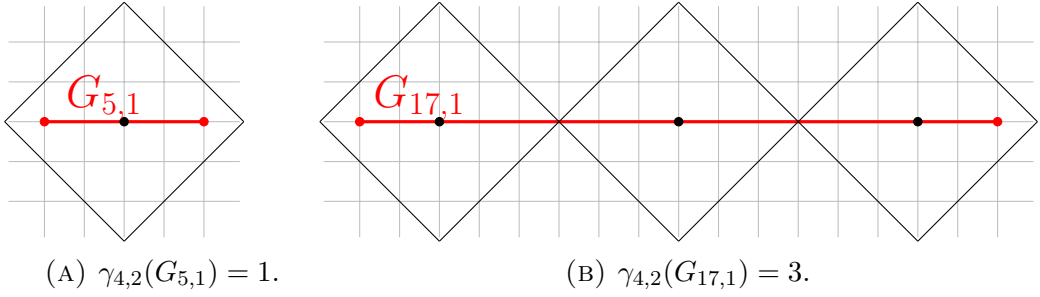

It follows that
\begin{equation}
\gamma_{t,2}(G_{m,1}) \leq \left \lfloor \frac{m+2(t-1)}{2(t-1)} \right \rfloor
= \left \lfloor \frac{mt - m + 2t^2 - 4t + 2}{2(t-1)^2} \right \rfloor.
\label{eq_leftside}
\end{equation}

To establish this result, we must show that
\begin{equation}
\gamma_{t,2}(G_{m,1})\leq\left \lfloor \frac{(m\!+\!2(t\!-\!2))(1\!+\!2(t\!-\!2))}{2(t\!-\!1)^2} \right \rfloor=\left \lfloor \frac{2mt\!-\!3m\!+\!4t^2\!-\!14t\!+\!12}{2(t\!-\!1)^2} \right \rfloor.
\label{eq_rightside}
\end{equation}
Comparing the numerators of Equations  \ref{eq_leftside} and  \ref{eq_rightside}, we find that when $m > 1$ and $t>2$, and when $m \geq 1$ and $t > 3$,
\[mt - m + 2t^2 - 4t + 2 \leq 2mt-3m+4t^2-14t+12.\]
Finally, in the $m=1$, $t=3$ case, we have that 
\begin{equation}
    \left \lfloor \frac{mt - m + 2t^2 - 4t + 2}{2(t-1)^2} \right \rfloor = \left \lfloor \frac{2mt-3m+4t^2-14t+12}{2(t-1)^2} \right \rfloor = 1.
\end{equation} Thus, the result holds for all $m \geq 1$, $t > 2$.
\end{proof}

Our next result establishes that when $m, n > 1$, $G_{m,n}$ can be dominated by a subset of a rectilinear $(t,2)$ broadcast covering an area only slightly larger than $m \times n$.

Let $G_{m,n} = (V_G, E_G)$ be the grid graph with dimensions $m \times n$ induced on $G_\infty$ by the vertices $\{ v_{(i,j)} \ | \ 0 \leq i < m, 0 \leq j < n\}$. Similarly, let $H_{m,n} = (V_H, E_H)$ be the grid graph with dimensions $(m+2(t-2)) \times (n+2(t-2))$ induced on $G_\infty$ by the vertices $\{ v_{(i,j)} \ | \ -(t-2) \leq i < m+(t-2), -(t-2) \leq j < n+(t-2)\}$. Thus $G_{m,n}$ is centered inside $H_{m,n}$. Let $\mathbb{T}$ be a rectilinear $(t,2)$ broadcast on $G_{\infty}$. For an example see Figure \ref{fig:GHoverlay}, which illustrates $G_{12,6}$, $H_{12,6}$, and $\mathbb{T}$ in the $t=3$ case.

\begin{figure}[h]
    \centering
    \begin{tikzpicture}[scale=0.8]
        \clip (-1.05, -1.05) rectangle (16,10);
        \grid[lightgray]{15}{9}
        
        \node[red] at (2.9,4.5) {\Large $G_{12,6}$};
        \draw[red, very thick] (2,2) -- (2,7);
        \draw[red, very thick] (2,7) -- (13,7);
        \draw[red, very thick] (13,7) -- (13,2);
        \draw[red, very thick] (13,2) -- (2,2);
        
        \node[orange] at (0.1,5) {\Large $H_{12,6}$};
        \draw[orange, very thick](1,1) -- (1,8);
        \draw[orange, very thick](1,8) -- (14,8);
        \draw[orange, very thick](14,8) -- (14,1);
        \draw[orange, very thick] (14,1) -- (1,1);
        
        \vertex[blue]{0}{4}{}
        \vertex[blue]{2}{6}{}
        \vertex[blue]{0}{0}{}
        \vertex[blue]{2}{2}{}
        \vertex[blue]{4}{4}{}
        \vertex[blue]{6}{6}{}
        \vertex[blue]{4}{0}{}
        \vertex[blue]{6}{2}{}
        \vertex[blue]{8}{4}{}
        \vertex[blue]{10}{6}{}
        \vertex[blue]{8}{0}{}
        \vertex[blue]{10}{2}{}
        \vertex[blue]{12}{0}{}
        
        \vertex[blue]{16}{0}{}
        \vertex[blue]{14}{2}{}
        \vertex[blue]{12}{4}{}
        \vertex[blue]{16}{4}{}
        \vertex[blue]{14}{6}{}
        
        \vertex[blue]{0}{8}{}
        \vertex[blue]{4}{8}{}
        \vertex[blue]{8}{8}{}
        \vertex[blue]{12}{8}{}
        \vertex[blue]{16}{8}{}
        
        \vertex[blue]{2}{10}{}
        \vertex[blue]{6}{10}{}
        \vertex[blue]{10}{10}{}
        \vertex[blue]{14}{10}{}
    \end{tikzpicture}
    \caption{ $G_{12,6}$, $H_{12,6}$, and rectilinear $(3,2)$ broadcast $\mathbb{T}$. }
    \label{fig:GHoverlay}
\end{figure}
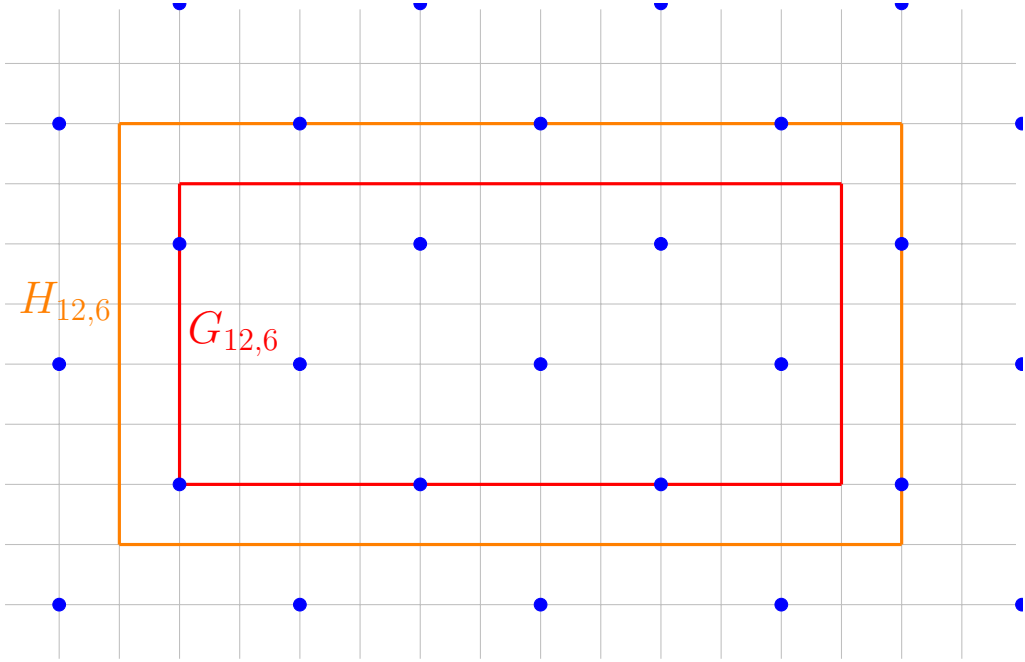

We now establish that the subset of $\mathbb{T}$ within $H_{m,n}$ is sufficient to dominate $G_{m,n}$.

\begin{lemma}
Let $m$, $n$, and $t$ be integers such that $m, n > 1$ and $t>2$. Let $\mathbb{T}$ be a rectilinear broadcast on $G_\infty$, and let $G_{m,n}$ and $H_{m,n}$ be grid graphs embedded in $\mathbb{Z} \times \mathbb{Z}$ as described previously. Then the set of tower vertices $V_H \cap \mathbb{T}$ supplies at least 2 signal to each vertex $v_{(i,j)} \in V_G$.
\label{lemma:VH}
\end{lemma}
\begin{proof}
Let $v_{(i,j)}$ be a vertex in $V_G$. We consider each possible location of $v_{(i,j)}$ with respect to the broadcast outlines of tower vertices in $\mathbb{T}$.

Suppose first that $v_{(i,j)}$ is contained inside the broadcast outline of a tower vertex $T \in \mathbb{T}$, in which case $dist(T, v_{(i,j)}) \leq t-2$ and $sig(T,v)\geq2$. As every vertex $u$ with $dist(u,v_{(i,j)}) \leq t-2$ is contained in $V_H$, $T \in V_H$ and $V_H \cap \mathbb{T}$ supplies at least 2 signal to $v_{(i,j)}$.

Suppose that $v_{(i,j)}$ is located on the broadcast outlines of exactly two towers $T_{(x_0, y_0)}$ and $T_{(x_1, y_1)}$ in $\mathbb{T}$. As $\mathbb{T}$ is rectilinear, this condition implies that $v_{(i,j)}$ is located along a straight edge of each broadcast outline, not at a corner. As $|i-x_0| + |j-y_0| = |i-x_1| + |j-y_1| = t-1$, and $v_{(i,j)}$ is not located at the corner of a broadcast outline, we have that
\[|i-x_0|, |i-x_1|, |j-y_0|, |j-y_1| < t-1.\] 
Thus, we have that $T_{(x_0, y_0)}, T_{(x_1, y_1)} \in V_H$ and $V_H \cap \mathbb{T}$ supplies at least 2 signal to $v_{(i,j)}$.

Finally, suppose that $v_{(i,j)}$ is located at the corner of four broadcast outlines. As $\mathbb{T}$ is rectilinear, the towers corresponding to these broadcast outlines are located at the points $(i-(t-1), j)$, $(i+(t-1), j)$, $(i, j-(t-1))$, and $(i, j+(t-1))$. Because $m, n > 1$ by assumption, we have that at least two of these towers must be contained within $H_V$, which ensures that $V_H \cap \mathbb{T}$ supplies at least 2 signal to $v_{(i,j)}$.
\end{proof}

With Lemmas \ref{lemma:linegraph} and \ref{lemma:VH} in hand, we are now prepared to prove our main result.

\mainthm*
\begin{proof}
Lemma \ref{lemma:linegraph} establishes Theorem \ref{thm_main} in the case where $m=1$ or $n=1$. We proceed with the assumption that $m, n > 1$.

Let $\mathbb{T}$ be a rectilinear broadcast on $G_\infty$, and let $G_{m,n}$ and $H_{m,n}$ be grid graphs embedded in $\mathbb{Z} \times \mathbb{Z}$ as described previously. By Lemma \ref{lemma:VH}, the vertex set $V_H \cap \mathbb{T}$ supplies at least 2 signal to each vertex in $V_G$.

We transform $V_H \cap \mathbb{T}$ into a $(t,2)$ broadcast on $G_{m,n}$ by replacing each vertex in $(V_H \setminus V_G) \cap \mathbb{T}$ with a corresponding vertex in $V_G$. Let $B$ denote the set that we will transform into our $(t,2)$ broadcast. To begin, set $B$ equal to $(V_H \setminus V_G) \cap \mathbb{T}$. For each vertex $v \in (V_H \setminus V_G) \cap \mathbb{T}$, there exists a unique vertex $v' \in V_G$ that minimizes $dist(v,v')$. For each vertex $v \in (V_H \setminus V_G) \cap \mathbb{T}$, remove $v$ from $B$ and replace it with $v'$. We claim that this operation leaves the cardinality of $B$ unchanged and preserves the property that the vertices in $B$ supply at least 2 signal to each vertex in $V_G$. Figure \ref{fig:replacement} illustrates the replacement operation for $G_{12,6}$ in the $t=4$ case.

\begin{figure}[h!]
    \centering
    \begin{tikzpicture}[scale=0.8]
        \clip (-1.05, -1.05) rectangle (16,10);
        \grid[lightgray]{15}{9}
        
        \node[red] at (2.9,4.5) {\Large $G_{12,6}$};
        \draw[red, very thick] (2,2) -- (2,7);
        \draw[red, very thick] (2,7) -- (13,7);
        \draw[red, very thick] (13,7) -- (13,2);
        \draw[red, very thick] (13,2) -- (2,2);
        
        \node[orange] at (0.8,5) {\Large $H_{12,6}$};
        \draw[orange, very thick](0,0) -- (0,9);
        \draw[orange, very thick](0,9) -- (15,9);
        \draw[orange, very thick](15,9) -- (15,0);
        \draw[orange, very thick] (15,0) -- (0,0);

        \vertex[blue]{3}{6}{}
        \vertex[blue]{6}{3}{}
        \vertex[blue]{9}{6}{}
        \vertex[blue]{12}{3}{}

        \vertex[cyan]{0}{3}{}
        \vertex[cyan]{3}{0}{}
        \vertex[cyan]{9}{0}{}
        \vertex[cyan]{0}{9}{}
        \vertex[cyan]{6}{9}{}
        \vertex[cyan]{12}{9}{}
        \vertex[cyan]{15}{6}{}
        \vertex[cyan]{15}{0}{}
        \draw[cyan, thick] (2,3) -- (0,3);
        \draw[cyan, thick] (3,2) -- (3,0);
        \draw[cyan, thick] (9,2) -- (9,0);
        \draw[cyan, thick] (2,7) -- (0,9);
        \draw[cyan, thick] (6,7) -- (6,9);
        \draw[cyan, thick] (12,7) -- (12,9);
        \draw[cyan, thick] (13,6) -- (15,6);
        \draw[cyan, thick] (13,2) -- (15,0);
        \vertex[blue]{2}{3}{}
        \vertex[blue]{3}{2}{}
        \vertex[blue]{9}{2}{}
        \vertex[blue]{2}{7}{}
        \vertex[blue]{6}{7}{}
        \vertex[blue]{12}{7}{}
        \vertex[blue]{13}{6}{}
        \vertex[blue]{13}{2}{}
    \end{tikzpicture}
    \caption{ Vertices in $(V_H \setminus V_G) \cap \mathbb{T}$ are replaced with counterpart vertices in $V_G$ to create a $(4,2)$ broadcast on $G_{12,6}$.}
    \label{fig:replacement}
\end{figure}
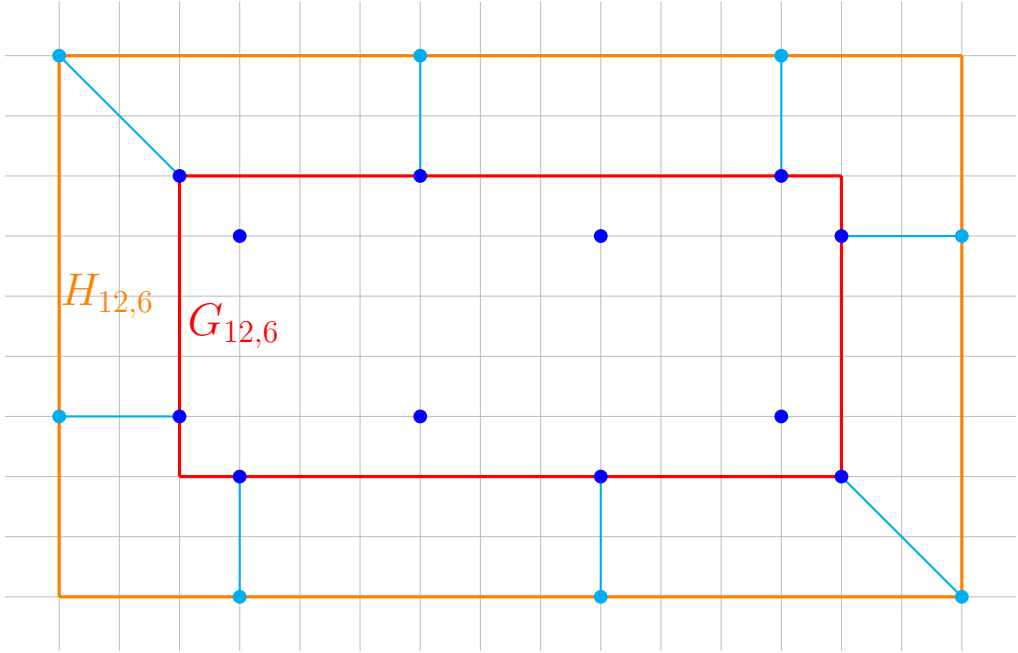

To prove our claim, we establish that the replacement operation maps each vertex $v \in (V_H \setminus V_G) \cap \mathbb{T}$ to a vertex $v' \not \in V_G \cap \mathbb{T}$ and that for any two vertices $u$ and $v$ in $(V_H \setminus V_G) \cap \mathbb{T}$, the replacement operation maps $u$ and $v$ to distinct vertices.

First, let $v$ be a vertex in $(V_H \setminus V_G) \cap \mathbb{T}$, and $v'$ be the vertex in $V_G$ that minimizes $dist(v,v')$. Because $\mathbb{T}$ is a rectilinear broadcast, for any pair of vertices $s, t \in \mathbb{T}$, we have $dist(s,t) \geq 2(t-1)$. However, $dist(v,v')$ is at most $2(t-2)$, which occurs when $v$ is located at a corner of $H_{m,n}$. Because $dist(v,v') < 2(t-1)$ and $v \in \mathbb{T}$, $v' \not \in \mathbb{T}$ and thus $v' \not \in V_G \cap \mathbb{T}$. 

Second, let $u$ and $v$ be two vertices in $(V_H \setminus V_G)$ such that the same vertex minimizes $dist(u,v')$ and $dist(v,v')$ over all vertices $v' \in V_G$. In this case, $dist(u,v)$ is at most $2(t-2)$, which occurs when $v'$ is located at a corner of $G_{m,n}$ and $u$ and $v$ are separated from $v'$ by horizontal and vertical paths of length $(t-2)$. Because $dist(u,v) < 2(t-1)$, it is not possible that both $u$ and $v$ are in $\mathbb{T}$. 

Thus, each time we replace a vertex $v \in (V_H \setminus V_G) \cap \mathbb{T}$ with the vertex $v' \in V_G$ that minimizes $dist(v,v')$, we are guaranteed that $v'$ is not already in $B$. Furthermore, for every vertex $v \in (V_H \setminus V_G) \cap \mathbb{T}$, the vertex $v' \in V_G$ that minimizes $dist(v, v')$ satisfies $dist(v,w) > dist(v',w)$ for every vertex $w \in V_G$. Thus, $B$ supplies at least as much signal to each vertex $v \in V_G$ as $V_H \cap \mathbb{T}$. $B$ is thus a $(t,2)$ dominating set for $G_{m,n}$.

We employ the probabilistic method to minimize $|B| = |V_H \cap \mathbb{T}|$ over all possible grid subgraphs $H$ with dimensions $(m + 2(t-2)) \times (n + 2(t-2))$. Pick a coordinate $(x,y)$ at random and determine a grid subgraph $H(x,y)$ by setting $v_{(x,y)}$ as its lower left corner. Because $\mathbb{T}$ is an optimal $(t,2)$ broadcast on the infinite grid, its broadcast density is $\frac{1}{2(t-1)^2}$ by Theorem \ref{thm_t2}. The expected value of $|V_{H(x,y)} \cap \mathbb{T}|$ is thus $\frac{(m + 2(t-2))( n + 2(t-2))}{2(t-1)^2}$, which implies the existence of a point $(x,y)$ for which
\[|V_{H(x,y)} \cap \mathbb{T}| = \left \lfloor \frac{(m + 2(t-2))( n + 2(t-2))}{2(t-1)^2} \right \rfloor = |B|.\]
Thus for $t>2$, and any grid graph $G_{m,n}$, there exists a $(t,2)$ broadcast of size at most $\left \lfloor \frac{(m + 2(t-2))( n + 2(t-2))}{2(t-1)^2} \right \rfloor$.
\end{proof}

Our result is nearly optimal. To demonstrate this, we derive the following lower bound as a corollary to Theorem \ref{thm_t2}.

\thmlb*
\begin{proof}
Suppose for contradiction that for some positive integers $m,n,$ and $t$,
\[\gamma_{t,2}(G_{m,n}) < \frac{mn}{2(t-1)^2}.\]
By assumption, there exists a set $S$ of vertices with $|S| < \frac{mn}{2(t-1)^2}$ that dominates $G_{m,n}$. Thus, we can generate a $(t,2)$ broadcast on $G_\infty$ by dividing $G_\infty$ into grids with dimensions $m \times n$ and selecting vertices that dominate each grid according to $S$.

$\mathbb{T}$ is thus an infinite $(t,2)$ broadcast with broadcast density less than $\frac{mn}{2(t-1)^2}$, which is contradictory by Theorem \ref{thm_t2}. 
\end{proof}

The upper bound of Theorem \ref{thm_main} and the lower bound of Theorem \ref{thm_2} grow at a rate quadratic in $m$ and $n$, while the difference between them grows at a rate linear in $m$ and $n$. Thus as $m$ and $n$ increase, the ratio of the two bounds approaches 1. When $t=3$, Theorem \ref{thm_main} states that
\begin{equation}
\gamma_{3,2}(G_{m,n}) \leq \left \lfloor \frac{(m+2)(n+2)}{8} \right \rfloor.
\end{equation}
This bound differs by a small constant factor from the upper bound on $\gamma_{3,2}(G_{m,n})$ provided by Blessing et al. \cite{Insko}. We thus confirm the conjecture of the authors that their bound is asymptotically optimal.

\section{Open Problems}

We conclude by providing several open problems and directions for future work.

\begin{itemize}
    \item In \cite{Insko} and \cite{grez2014new}, the authors systematically improve their bounds by a constant value of 4 by adjusting vertices at the corners of $G_{m,n}$. Are similar constant improvements possible for our bounds on $(t,2)$ broadcast domination numbers?
    \item In \cite{DHR}, the authors provide an upper bound on the density of optimal $(t,3)$ broadcasts on $G_\infty$. By a letterboxing method similar to that employed in this paper, this bound may be translated to an upper bound on $\gamma_{t,3}(G_{m,n})$. However, the original bound is not proven to be optimal, and the resulting bound for finite grids may be off by an amount proportional to the size of the grid. Does this bound approach optimality? Can it be improved?
\end{itemize}

\section*{Acknowledgements}

The author expresses his thanks to Pamela E. Harris for her feedback on earlier drafts of this manuscript, and to an anonymous reviewer for careful reading and helpful suggestions.

\printbibliography
\end{document}